\documentclass[10pt]{article}

\usepackage{graphicx,amssymb,amsmath}
\usepackage{epsfig}
\usepackage{epsf}
\usepackage{amsfonts}
\usepackage{latexsym}
\usepackage{amssymb}
\usepackage{amsthm}
\usepackage{graphics}
\usepackage{float}
\usepackage{subfigure}

\newtheorem{thm}{Theorem}[section]

\newtheorem{defi}[thm]{Definition}
\newtheorem{lem}[thm]{Lemma}

\newtheorem{claim}[thm]{Claim}

\title{Multiple coverings with closed polygons}

\author{%
Istv\'an Kov\'acs\thanks{Technical University, Budapest.
Supported by T\'amop - 4.2.2.B-10/1--2010-0009.}
\and
G\'eza  T\'oth\thanks{Alfr\'ed R\'enyi Institute of Mathematics,
Budapest, Hungary. 
Supported by Hungarian Science Foundation Grants OTKA K-83767 and NN-102029}}

\begin{document}
\maketitle

\begin{abstract}
A planar set $P$ is said to be cover-decomposable
if there is a constant $k=k(P)$ such that
every $k$-fold covering of the plane with translates of $P$
can be decomposed into two coverings. It is known that open convex polygons
are cover-decomposable. Here we show that 
closed, centrally symmetric convex polygons are 
also  cover-decomposable. 
We also show that an {\em infinite-fold} covering of the plane 
with translates of $P$
can be decomposed into two infinite-fold coverings.
Both results hold for coverings of any subset of the plane. 
\end{abstract}

\section{Introduction}

The study of multiple coverings was initiated by Davenport and L.
Fejes T\'oth about 60 years ago \cite{BMP05}. 
Let ${\cal S}=\{\ S_i\ |\ i\in I\ \}$ be a collection of sets in the plane.
We say that ${\cal S}$ is an {\em $m$-fold
covering} if every point of the plane is contained in at least
$m$ members of $\cal H$. 
A $1$-fold covering is simply called a
{\em covering}.
Clearly, the union of $k$ coverings is always a $k$-fold covering, but it is
easy to see that the converse 
is not necessarily true, not even in the special case when ${\cal S}$ is a collection of {\em
translates} of a given set.

\begin{defi}\label{def-coverdc}  
A planar set $S$ is said to be
{cover-decomposable} if there exists a (minimal) constant
$k=k(S)$ such that every $k$-fold covering of the plane with
translates of $S$ can be decomposed into two coverings.
\end{defi}

J. Pach proposed the
problem of determining all cover-decomposable sets in 1980 \cite{P80}. 
He conjectured that all planar convex sets are cover-decomposable. 
Today there is a vast
literature on this subject, partly because of its theoretical interest \cite{PPT14}, and
partly because of its applications in the  
{\em sensor cover problem} in sensor network
scheduling \cite{GV11}. 

Pach verified his conjecture for centrally symmetric
open convex polygons \cite{P86}.
The next result in this direction was by G. Tardos and 
G. T\'oth \cite{TT07}, they proved that open triangles 
are cover-decomposable.
Finally, D. P\'alv\"olgyi and G. T\'oth \cite{PT10}
proved that all open convex polygons are cover-decomposable.

Observe, that all of these general positive results hold only 
for {\em open} sets.
The reason is that -- based on the ideas of Pach \cite{P86} --
all proofs reduce the problem to a finite problem, and that reduction 
works only for open sets. 
We belive that in fact all these results can be generalized for the closed 
version. The main result of this paper is the first step in this direction.

\begin{thm}\label{fotetel}
Every centrally symmetric
closed convex polygon is cover-decomposable.
\end{thm}

From the other direction, J. Pach, G. Tardos, and G. T\'oth \cite{PTT07} 
proved that (open and closed) 
concave quadrilaterals are {\em not} cover-decomposable. 
It was generalized by D. P\'alv\"olgyi \cite{P10} who showed for a large class 
of concave polygons that they are not  cover-decomposable. 
It is still not known whether there exists a  cover-decomposable
concave polygon.
Very recently, D. P\'alv\"olgyi \cite{P13} {em refuted} Pach's conjecture. He
proved that open and closed  
sets
with smooth boundary are 
not cover-decomposable. In particular, the unit disc is not
cover-decomposable.

Splitting infinite-fold coverings can lead to very deep problems.
Elekes, M\'artai and Soukup \cite{EMS11} constructed an infinite-fold covering 
of the line by translates of a closed set, whose decomposability 
is independent of ZFC. We believe that it is not the case 
for coverings of the plane with translates of a {\em convex} closed set.

It follows directly from Theorem \ref{fotetel} that an infinite-fold covering
of the plane with translates of a closed, convex, centrally symmetric polygon
is
decomposable into two coverings. We prove the following stronger result.

\begin{thm}\label{vegtelen}
Let $S$ be a closed, convex, centrally symmetric polygon.
Then every infinite-fold covering of the plane with 
translates of $S$ can be decomposed into two infinite-fold coverings. 
\end{thm}

Cover-decomposability has many other versions, 
instead of the plane, we can investigate, and decompose 
coverings of an arbitrary subset of the plane. 
We can consider only coverings with finite, countably many, or arbitrarily
many translates. 
In the last section we review some of these versions of cover-decomposability.
Our Theorems \ref{fotetel} and \ref{vegtelen} hold for each of these
versions, with the same proof.

\section{Centrally symmetric closed convex polygons;\\
Proof of Theorem 1.}

\subsection{Taking the dual, reduction to wedges}

Just like in most of the papers about cover-decomposability, 
 we formulate
and solve the problem in its {\em dual
form}. The idea is originally due to J. Pach 
\cite{P86}.
Suppose that $S$ is an open or closed, centrally symmetric convex polygon,
its vertices are $v_1, v_2$, $\ldots$, $v_{2n}$, ordered clockwise. Indices are
understood modulo $2n$. 

\begin{defi}
For any two points, $a$ and $b$, let
$\overrightarrow{ab}$ denote the halfline whose endpoint is $a$ and goes
through $b$.
Let $arg(\overrightarrow{ab})$ denote the clockwise angle from the positive
$x$ axis to $\overrightarrow{ab}$.
\end{defi}

\begin{defi}\label{wedgedef}
For every 
$i$, $1\le i\le 2n$, let $E_i$ denote the convex wedge whose bounding 
halflines are the translates of $\overrightarrow{v_iv_{i-1}}$ and  
$\overrightarrow{v_iv_{i+1}}$.
If $S$ is closed (resp. open), then let $E_i$ also be closed (resp. open).
$E_i$ is called the wedge that belongs to vertex $v_i$ of $S$.
We say that a wedge $E$ belongs to $S$, or $E$ is an $S$-wedge, 
if it belongs to one of its vertices.
For any point $p$, let $E_i(p)$ denote the translate of $E_i$ such that its
apex is in $p$.
\end{defi}

Now we can state the dual version of Theorem \ref{fotetel}

\begin{thm}\label{dualfotetel}
Let $S$ be a centrally symmetric
closed convex polygon, with vertices $v_1, v_2, \ldots , v_{2n}$, ordered
clockwise. 
Then there is an $m=m(S)>0$ with the following property.

Any bounded point set $\cal H$ can be colored with red and blue such that 
any translate of an $S$-wedge, $E_i(p)$, if 
$|E_i(p)\cap{\cal H}|\ge m$, then $E_i(p)\cap{\cal H}$ contains points of both
colors.
\end{thm}

\medskip

\noindent {\bf Proof of Theorem \ref{fotetel} from Theorem \ref{dualfotetel}.}

Let $x=x(S)$ be a number with 
the property that a square of side $x$ intersects at most two consecutive
sides of $S$. 
Divide the plane into squares of side $x$, by a square grid.
There is a constant $k'$ such that any translate of $S$ intersects at most
$k'$ little squares.

For any point $p$, let $S(p)$ denote the translate of $S$ so that its center
is at $p$. 
Let ${\cal H}=\{\ S_i\ |\ i\in I\ \}$ be a collection of translates of $S$
that form a $k=k'm$-fold covering, where 
$m=m(S)$ from Theorem \ref{dualfotetel}.
For every $i\in I$, let $c_i$ be the center of $S_i$.
Let ${\cal H}'=\{\ c_i\ |\ i\in I\ \}$ be the set of centers.
For any point $a$, $a\in S_i$ if and only if $c_i\in S(a)$.
Therefore, for every point $a$, $S(a)$ contains at least $k$ points of 
${\cal H}'$. 

The collection ${\cal H}$
can be decomposed into two
coverings if and only if the set 
${\cal H}'$ 
can be colored with
two colors, such that  every translate of $S$ contains a point of both
colors.

Color the points of  ${\cal H}'$ in each square separately, 
satisfying the conditions of Theorem \ref{dualfotetel}.
Now return to the covering ${\cal H}$ and color each 
translate of $S$ in  ${\cal H}$ to the color of its center. 
We claim that both the red and the blue translates form a covering.
Let $p$ be an arbitrary point, we have to show that it is covered by a
translate of both colors.
Or equivalently, $S(p)$ contains a point of  ${\cal H}'$ of both colors.
Since $S(p)$ contains at least $k$ points of 
${\cal H}'$, 
it contains at least $k/k'=m$ points in one of the little squares $Q$. 
But $S(p)$ intersects $Q$ ``like a
wedge'' that is, $S(p)\cap Q=E(q)\cap Q$ for some $S$-wedge $E$ and point $q$.
Therefore, by Theorem \ref{dualfotetel}, $S(p)\cap Q$ contains a point of
${\cal H}'$ of both
colors.
$\Box$

\smallskip

Now we ``only''
have to prove Theorem \ref{dualfotetel}. We need some preparation.

\subsection{Some properties of boundary points}

Theorem \ref{dualfotetel} has been proved by J. Pach \cite{P86} in the special
case when ${\cal H}$ is finite. Some parts of our proof are just modifications
of his argument, but some other parts are completely new.

Let $S$ be a centrally symmetric, open or 
closed convex polygon, its vertices 
 $v_1, v_2, \ldots , v_{2n}$ in clockwise direction, the $S$-wedges are
$E_1, E_2, \ldots , E_{2n}$, respectively.
Let $\cal H$ be a bounded point set.

\smallskip

\begin{defi}\label{i-boundary}
If $S$ is closed (resp. open),  
a point $p\in{\cal H}$ is called an
$E_i$-boundary point if 
$E_i(p)\cap{\cal H}=\{ p\}$ (resp. $E_i(p)\cap{\cal H}=\emptyset$).

Let $\mathbb{B}_i=\mathbb{B}_i({\cal H})$ denote the set of $E_i$-boundary
points of  ${\cal H}$.

Let $\mathbb{B}=\mathbb{B}({\cal H})=\cup_{i=1}^{\infty}\mathbb{B}_i$ denote
the set of all boundary points of ${\cal H}$, it is called the boundary of
${\cal H}$.
The other points of ${\cal H}$ are called interior points.
\end{defi}

For every $i$, $1\le i\le 2n$, we introduce 
an ordering of the $E_i$-boundary $\mathbb{B}_i$. These orders together will
give a cyclic ordering of the boundary  $\mathbb{B}$, where some boundary
vertices
appear twice. 
Let $\ell_i$ be a line perpendicular to the angular bisector of $E_i$. Direct
$\ell_i$
so that $E_i$ can be translated to the {\em left} side of
$\overrightarrow{\ell_i}$.
There is a natural ordering of the points of $\overrightarrow{\ell_i}$.
For $x, y\in \overrightarrow{\ell_i}$ we say the $x$ precedes $y$ 
($y$ follows $x$) if the vector $\overrightarrow{xy}$ points to the same
direction as
$\overrightarrow{\ell_i}$.
Orthogonally project the points of $\mathbb{B}_i$ to
$\overrightarrow{\ell_i}$, the image of $p$ is $\pi(p)$

It is easy to see that the map $\pi$ is injective.
If $p_1$, $p_2\in \mathbb{B}_i$ 
and $\pi(p_1)=\pi(p_2)$, then either $p_2\in E_i(p_1)$ or 
 $p_1\in E_i(p_2)$, 
but both of them are impossible since both $p_1$ and $p_2$ are 
$E_i$-boundary points.

\begin{defi}
Let $1\le i\le 2n$. For any two $E_i$-boundary points $p_1$ and $p_2$, 
let $p_1\prec_i p_2$ if and only if $\pi_i(p_1)$ precedes 
 $\pi_i(p_2)$ on $\overrightarrow{\ell_i}$.
\end{defi} 
 
The relation $\pi_i$ is a linear ordering on  $\mathbb{B}_i$.
Based on $\pi_i$,
we can define intervals on $\mathbb{B}_i$, for example,
$$[p_1,p_2]=\{p\in \mathbb{B}_i : \pi_i(p)\in [\pi_i(p_1), \pi_i(p_2)]\}.$$
We say that the {\em first half} of the interval $[p_1,p_2]$ is
$$\left\{ p\in \mathbb{B}_i : \pi_i(p)\in \left[ \pi_i(p_1),
    {\pi(p_1)+\pi_i(p_2)\over 2}\right] \right\}.$$
We define  $(p_1,p_2)$ similarly as $[p_1,p_2]$, and we define 
the {\em second half} of an interval similarly as the first half.

\begin{claim}\label{i,i+1hatar}
Suppose that $p\in\mathbb{B}_i$ and $p\in\mathbb{B}_{i+1}$, that is, 
$p$ is a boundary point with respect to both $E_i$ and $E_{i+1}$.
Let $\ell$ be the line through $p$, parallel to $v_iv_{i+1}$.
Then one of the closed halfplanes defined by $\ell$ contains all points of
${\cal H}$.
\end{claim}

\begin{proof}
If $p$ is a boundary point with respect to both $E_i$ and $E_{i+1}$, then 
$E_i(p)\cap{\cal H}=E_{i+1}(p)\cap{\cal H}=\emptyset$ if $S$ is open, and 
$\{ p\}$ if $S$ is closed. 
But $E_i(p)\cup E_{i+1}(p)$ contains an open halfplane bounded by $\ell$. 
This halfplane does contain any point of ${\cal H}$, therefore, its complement
satisfies the conditions.
\end{proof}

It is easy now that if $S$ is closed, then there is at most one point
$p\in{\cal H}$ that is a boundary point with respect to both $E_i$ and
$E_{i+1}$.
If $S$ is open and both $p$ and $q$ are such  boundary points,
then $p\prec_iq$ if and only if  $p\prec_{i+1}q$.
It also follows from Claim \ref{i,i+1hatar} that if
$p$ is a boundary point with respect to both $E_i$ and
$E_{i+1}$ and  
$q$ is a boundary point with respect to $E_i$ but not 
$E_{i+1}$, then  $p\prec_iq$.

There could be other types of boundary points with respect to more than one 
wedge.

\begin{defi}
A point $p\in{\cal H}$ is a singular boundary point if there are numbers 
$1\le i_1<n_1<i_2<n_2\le 2n$, or 
$1\le n_1<i_1<n_2<i_2\le 2n$ such that
$p$ is a boundary point with respect to $E_{i_1}$ and
$E_{i_2}$, but not a boundary point with respect to $E_{n_1}$ and
$E_{n_2}$,  see in Figure \ref{singular}.
Non-singular boundary points are called regular boundary points.
\end{defi}

\begin{figure}[H]
\begin{center}
\includegraphics[height=40mm]{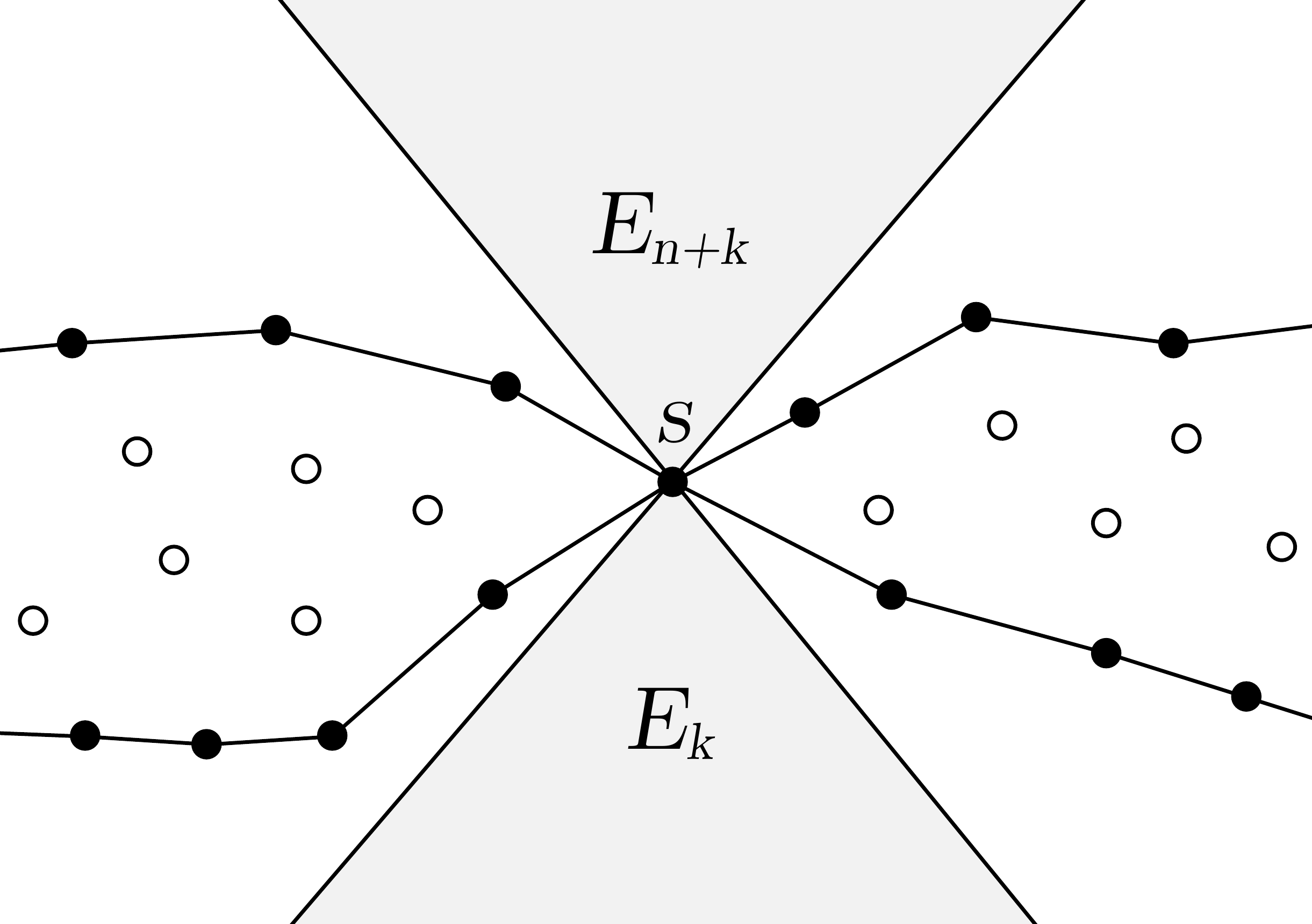}
\caption{$s$ is a singular boundary point}
\label{singular}
\end{center}
\end{figure}

This concept, just like the next two claims, are basically again from
\cite{P86}.

\begin{claim}\label{E_iE_{i+n}}
If $p$ is a singular $E_i$-boundary point, then 
it is a boundary point with respect to  $E_i$  and $E_{i+n}$ (the reflection
of $E_i$) and no other wedge.
\end{claim}

\begin{proof}
Suppose that $1\le i_1<n_1<i_2<n_2\le 2n$,
$p$ is a boundary point with respect to $E_{i_1}$ and
$E_{i_2}$, but it is not a 
 boundary point with respect to $E_{n_1}$ and
$E_{n_2}$, and $i_1+n\neq i_1$. Assume wlog. that $i_1=1$, $i_1=k<n$. 
Then $E_1(p)$ and $E_k(p)$ do not contain any point of ${\cal H}$, different
from $p$.
It follows from the convexity of $S$, that 
$E_{n_1}(p)\subset E_1(p)\cup E_k(p)$, therefore, $p$ 
 is a boundary point with respect to $E_{n_1}$,
a contradiction, see in Figure \ref{singpprop}a.
The argument is the same if we have
$1\le n_1<i_1<n_2<i_2\le 2n$.
\end{proof}

Now we show the all singular boundary points are of the ``same type''.

\begin{claim}\label{egyfele_szing}
If $p$ is a singular boundary point with respect to $E_i$ and $E_{i+n}$, then 
there is no singular boundary point with respect to some other pair of wedges.
\end{claim}

\begin{proof}
Suppose that $p$ and $q$ are singular boundary points with respect to
different pairs of wedges, say, $p$ with respect to $E_1$ and $E_{n+1}$, 
$q$ with respect to $E_k$ and $E_{n+k}$, $1<k\le n$. 
It follows that either 
\begin{equation}\label{v_1v_2pq}
arg(\overrightarrow{v_1v_2})\le arg(\overrightarrow{pq})\le
arg(\overrightarrow{v_{2n}v_1}),
\end{equation}
or
\begin{equation}\label{v_1v_2qp}
arg(\overrightarrow{v_1v_2})\le arg(\overrightarrow{qp})\le
arg(\overrightarrow{v_{2n}v_1}).
\end{equation}
Suppose wlog. that  \eqref{v_1v_2pq} holds. Since $q$ is a boundary point with
respect to  $E_k$ and $E_{n+k}$, 
\begin{equation}\label{v_{k-1}v_kpq}
arg(\overrightarrow{v_{k-1}v_{k}})\le arg(\overrightarrow{pq})\le arg(\overrightarrow{v_{k}v_{k+1}}).
\end{equation}
(Note, that if $S$ is closed then the above inequalities are strict
inequalities.)
But \eqref{v_1v_2pq} and \eqref{v_{k-1}v_kpq} can simultaneously only if  
$k=2$ and  
$$arg(\overrightarrow{v_1v_2})=arg(\overrightarrow{pq}).$$
But in this case, $q$ is also an $E_1$-boundary point, so it is not singular,
a
contradiction see in Figure \ref{singpprop}b. 
\end{proof}

\begin{figure}[H]
\centering
\subfigure[Singular point there is only with opponent wedge pair]
{
\includegraphics[height=40mm]{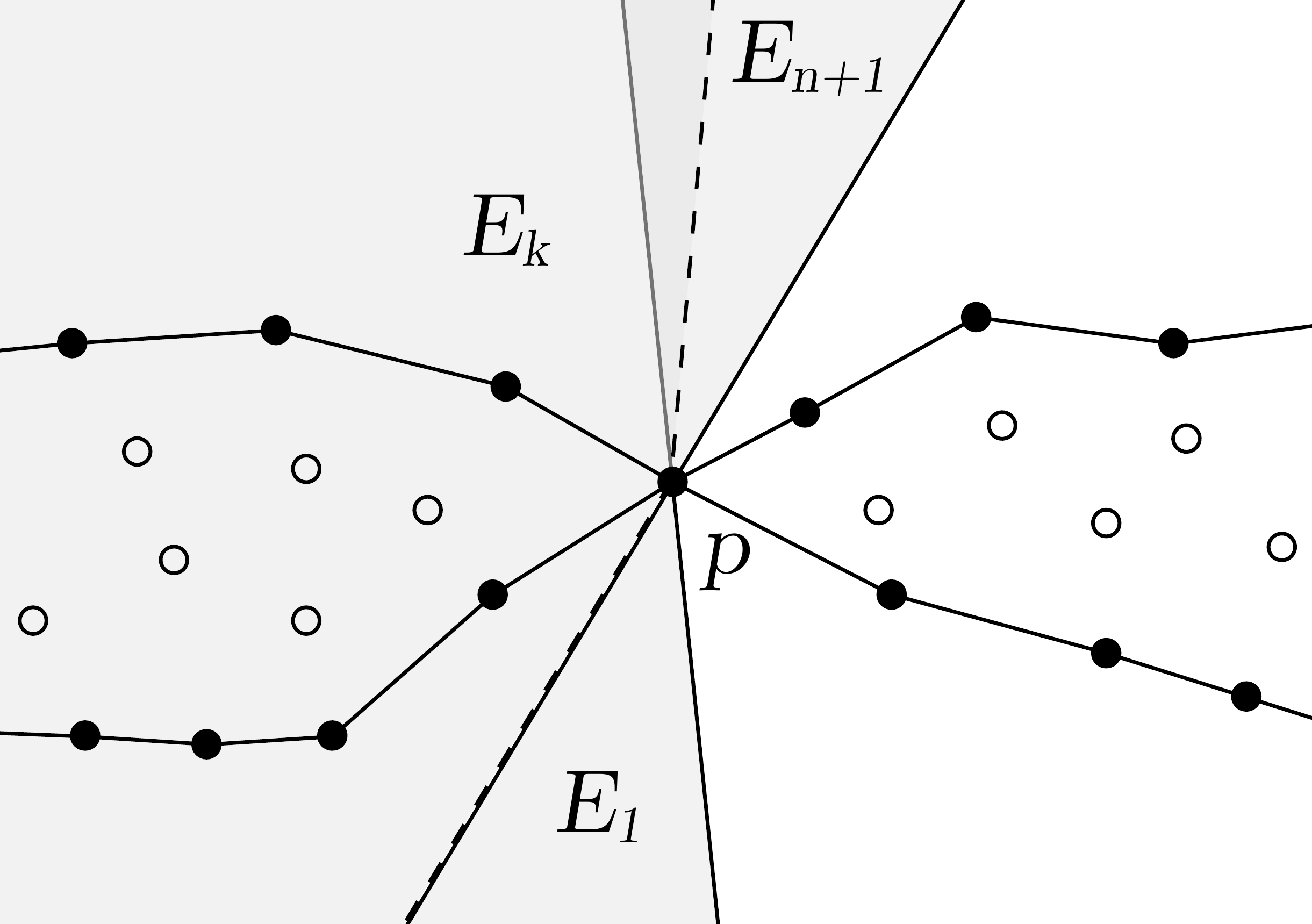}
}
\subfigure[All singular points are same type]
{
\includegraphics[height=40mm]{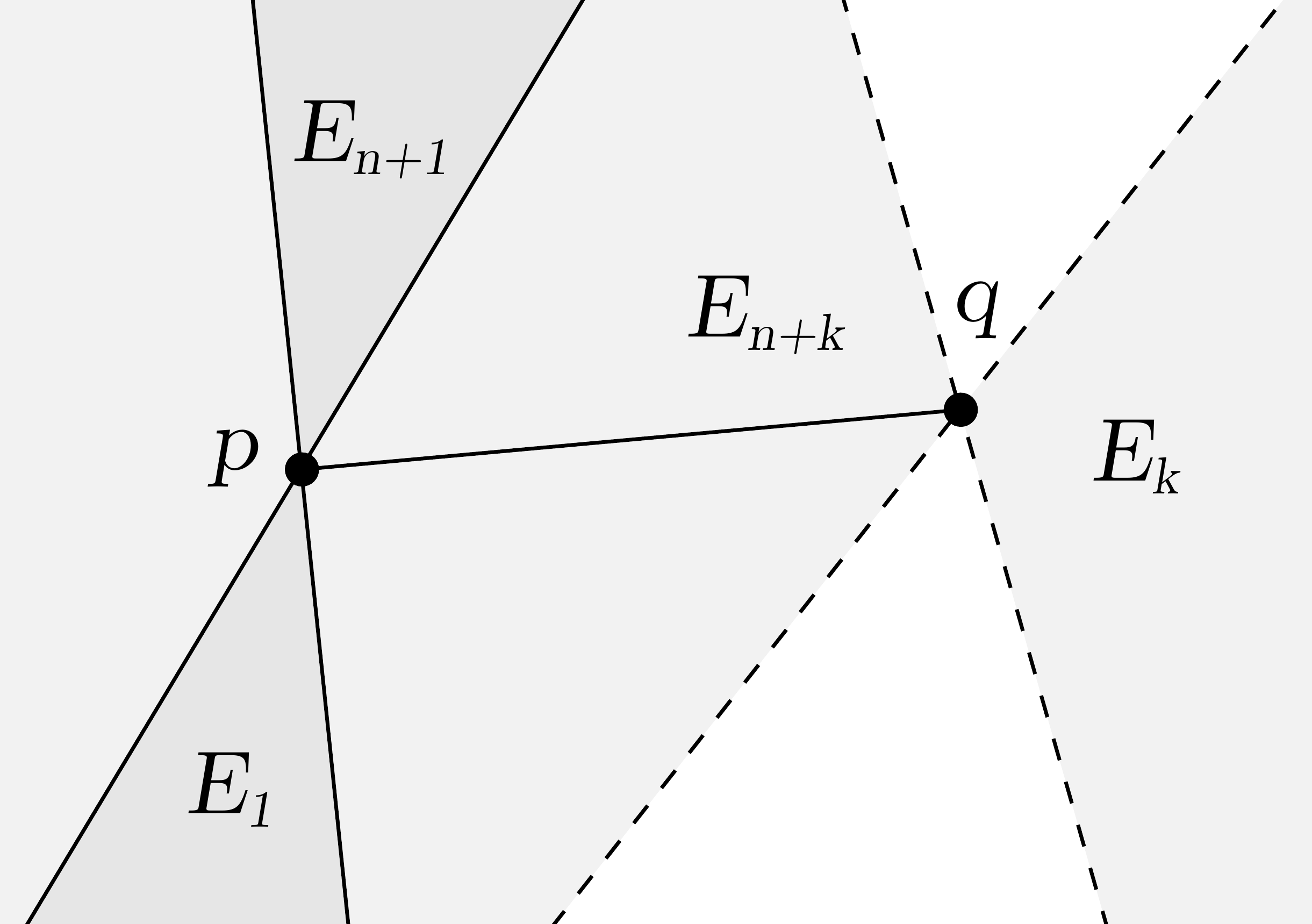}
}
\caption{Properties of singular boundary points}\label{singpprop}
\end{figure}

From now on, suppose, without loss of generality, that all singular boundary points 
of ${\cal H}$
are $E_1$- and $E_{n+1}$-boundary points. 
Observe, that if $p$ and $q$ are singular boundary points with respect to
$E_1$ and $E_{n+1}$, then $p\prec_1 q \Leftrightarrow q\prec_{n+1} p$.
The {\em type} of a boundary point $p$ is the smallest $i$ such that 
$p$ is an $E_i$-boundary point. 

In the set 
$\mathbb{B}$
of  boundary points, substitute each singular boundary point $p$ by 
$p'$ and $p''$, such that $p'$ is an $E_1$-boundary point, $p''$ is an
$E_{n+1}$-boundary point. Let $\mathbb{B}'$ be the resulting set. 
For $p$, $q\in \mathbb{B}'$, let $p\prec q$ if

\begin{itemize}

\item $p$ is of type $i$, $q$ is of type $j$, and $1 \leq i<j \le 2n$,

\item Both $p$ and $q$ are of type $i$, and $p\prec_i  q$.

\end{itemize}

Relation $\prec$ gives a linear ordering on 
$\mathbb{B}'$.
We have the elements in the following order:

\begin{itemize}

\item Boundary points with respect to both $E_{2n} $ and $E_1$, ordered 
according to $\prec_{2n}$ and $\prec_1$;

\item $E_1$-boundary points, ordered according to $\prec_1$;

\item Boundary points with respect to both $E_{1} $ and $E_2$, ordered 
according to $\prec_{1}$ and $\prec_2$;

\item $E_2$-boundary points, ordered according to $\prec_2$;

\item Boundary points with respect to both $E_{2} $ and $E_3$, ordered 
according to $\prec_{2}$ and $\prec_3$;

\item $E_3$-boundary points, ordered according to $\prec_3$;

\item $\ldots$

\item $E_{2n}$-boundary points, ordered according to $\prec_{2n}$.

\end{itemize}

If we project the points of $\mathbb{B}'$ on a circle, then there is a natural
way to define {\em intervals} on  $\mathbb{B}'$, and then also on 
$\mathbb{B}$. 
No we define them precisely.

\begin{defi}
An $I\subset \mathbb{B}'$ subset is called an {\em interval} of $\mathbb{B}'$,
if one of the following two conditions hold.

(i) If $p\prec q\prec r$ and $p, r\in I$, then $q\in I$.

\smallskip

(ii) If $p\prec r$, $p, r\in I$, and either $q\prec p$ or $r\prec q$, then $q\in I$.

A subset $I\subset \mathbb{B}$ is  called an {\em interval} of $\mathbb{B}$ if
the corresponding subset 
$I'\subset \mathbb{B}'$ is an {\em interval} of $\mathbb{B}'$.

An interval of  $\mathbb{B}$ or  $\mathbb{B}'$ is 
called {\em homogeneous} if all its points are $E_i$-boundary points, for some
$i$.
\end{defi}

\begin{claim}\label{2interval}
A translate of an $S$-wedge $E_i$ intersects  $\mathbb{B}$
in at most two intervals.
\end{claim}

\begin{proof}
Consider a translate of an $S$-wedge, say, $E_2(z)$. Suppose that $p$ is an
$E_i$-boundary point, $q$ is an $E_j$-boundary point, 
$3\le i,j\le n+1$, 
$p\in E_2(z)$ and $p\succ q$. Then
$$arg(\overrightarrow{v_2v_1})\le
arg(\overrightarrow{pq})\le
arg(\overrightarrow{v_2v_3}),$$
so 
$q\in E_2(z)$.
We can argue similarly if $p$ and $q$ are on the ``other side'', that is, 
$i, j\in \{ n+3, n+4, \ldots , 2n, 1 \}$.

Suppose now that $p$ is an $E_i$-boundary point, $q$ is an $E_2$-boundary point,
$3\le i\le n+1$, 
$p, q\in E_2(z)$ and $p\succ r\succ q$. Then again
$$arg(\overrightarrow{v_2v_1})\le
arg(\overrightarrow{pr})\le
arg(\overrightarrow{v_2v_3}),$$
therefore, $r\in E_2(z)$.
Again, we can argue similarly in the case 
$i\in \{ n+3, n+4, \ldots , 2n, 1 \}$.

Finally, suppose that 
$p$, $q$, and $r$ are $E_2$-boundary points, 
$p, r\in E_2(z)$ and $p\succ q\succ r$. Again, it is easy to check that 
$r\in E_2(z)$.
The same holds if 
$p$, $q$, and $r$ are $E_{n+2}$-boundary points.

It follows from these observations that $E_2(z)$ intersects  $\mathbb{B}$
in at most two intervals.
\end{proof}

\begin{figure}[H]
\begin{center}
\includegraphics[height=40mm]{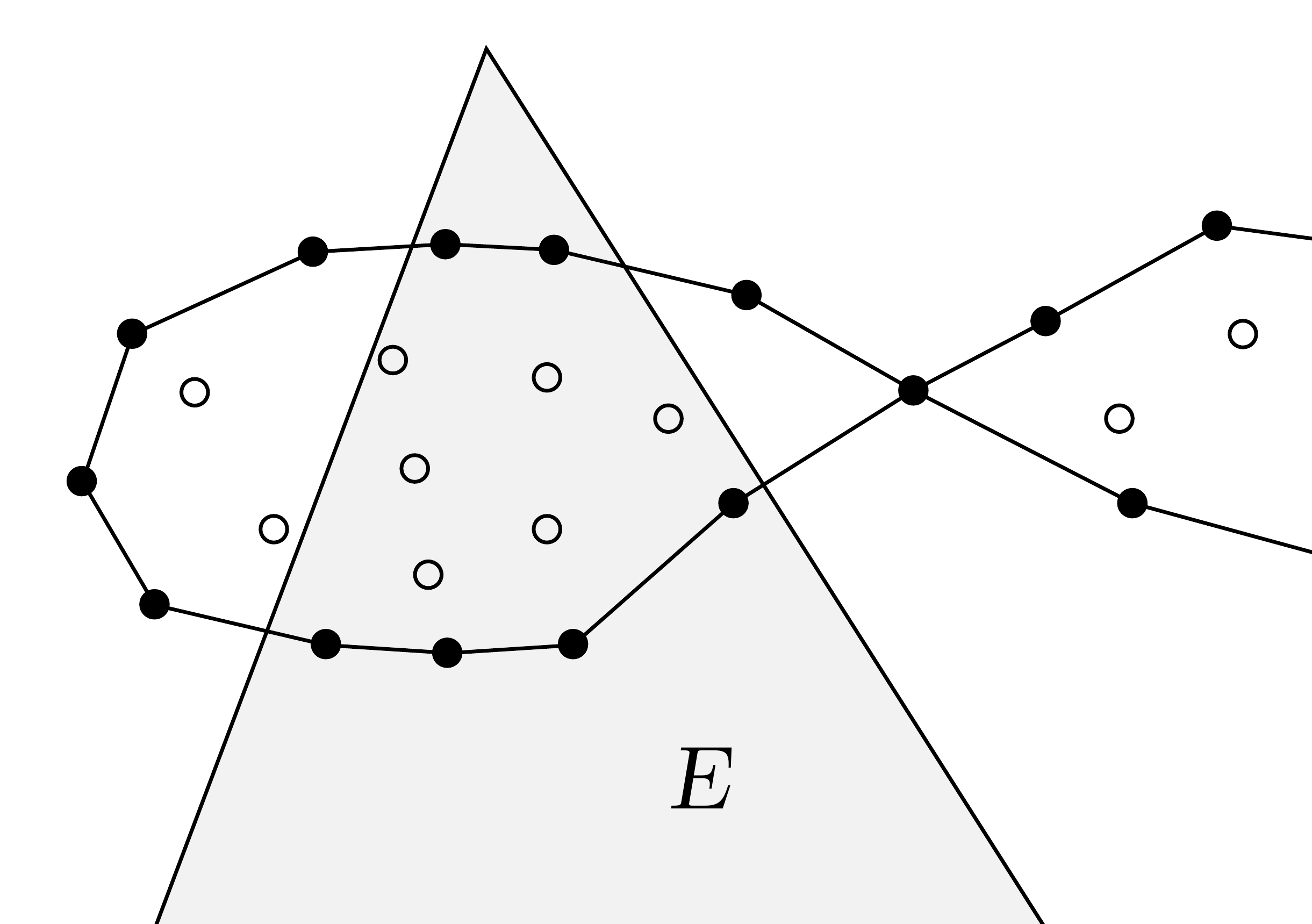}
\caption{$E$ intersects  $\mathbb{B}$
in at most two intervals}
\label{interval}
\end{center}
\end{figure}




\subsection{Coloring algorithm}

Two boundary points are neighbors if there is no other boundary point between
them. More precisely:

\begin{defi}
Two boundary points $p, q\in \mathbb{B}'$ are {\em neighbors} if $p\prec q$,
and either 
(i) there is no $r$ with $p\prec r\prec q$, 
or
(ii) there is no $r$ with $r\prec p$ or $q\prec r$.

Two boundary points $p, q\in \mathbb{B}$ are {\em neighbors} if the
corresponding points in $\mathbb{B}'$
are neighbors.
Let $p\sim q$ denote that $p$ and $q$ are neighbors.
\end{defi}

Let $\approx$ be the transitive closure of the relation $\sim$ on 
$\mathbb{B}$, that is, 
$p\approx q$ if and only if there is a finite sequence of boundary points,
starting with $p$, ending with $q$, such that the consecutive pairs are neighbors.
The relation $\approx$ is an equivalence relation.
Those boundary points $p$ which belong to an equivalence class of size one,
are called {\em lonely} boundary points. The others, which have a neighbor,
are called {\em social} boundary points.

First we give a coloring procedure which colors the points black and white. 
Then we apply it several times to obtain our red-blue coloring.

\bigskip

\noindent {\sc Black-White-Boundary-Coloring($S, {\cal H}$)}


\begin{itemize}

{\sl 

\item[] Divide the boundary of ${\cal H}$, $\mathbb{B}$, into equivalence
classes by relation $\approx$. 
First we color the social boundary points.

Let $C$ be an arbitrary equivalence class, $|C|>1$.
If $C$ contains singular boundary points, then color them first,
so that consecutive points receive different colors.
Then, if there are regular boundary points  
between two consecutive singular boundary points, 
color them so that no two consecutive boundary points are black and 
no three consecutive are white.
Do the same for each equivalence class  $|C|>1$.

Now we color the lonely boundary points, denote their set by 
$\mathbb{B}_{lonely}$. It is the union of at
most $2n$ homogeneous intervals, that is,
$\mathbb{B}_{lonely}=\cup_{i=1}^{2n}I_i$ where the elements of $I_i$ are all
$E_i$-boundary points. We color each interval separately.
Recall that, based on projection $\pi_i$, we defined the midpoint, the first
and the second half of a  homogeneous interval.

For each $i$, consider interval $I_i$. If it contains infinitely many points,
color one of them black. Then again, for each $i$, if $I_i$ contains
infinitely many points,
color an uncolored one white.
If it contains finitely many points, color all of them white.
Now half each interval which contained infinitely many points, and drop
intervals with finitely many points. Let 
$J_1$, $J_2$, $\ldots $, $J_m$ be the set of new intervals. Repeat the
previous step, choose an uncolored point in each of the intervals with
infinitely many points, and color them black, then 
the same with white, and then color all uncolored points in intervals with
finitely many points white. 
Repeat this infinitely many times.

Then, if there is still an uncolored point, color it white. 

}

\end{itemize}

\begin{claim}\label{mindketszin} 
If there are infinitely many points in an interval of $\mathbb{B}$, then it
contains infinitely many points of both colors. 
\end{claim}

\begin{proof}
We can assume that interval $I$ is homogeneous, say, all of its points are
$E_i$-boundary points.
Suppose first that $I$ contains a lonely boundary point $p$ in its interior.
Then there is an accumulation point $q$ of boundary points in the
interior of $I$. (Note that $q$ is not necessarily an element of ${\cal H}$.)
If $q$ is an accumulation point of lonely boundary points, then 
our procedure {\sc Black-White-Boundary-Coloring($S, {\cal H}$)}
will arrive to an interval $J\subset I$ which contains infinitely many
lonely boundary points, and it colors one of them white, one black. 
Moreover, the procedure will find such an interval in infinitely many steps, so
it colors infinitely many points white, and infinitely many black.

If $q$ is an accumulation point of social boundary points, or if $I$ does not
contain
  a lonely boundary point $p$ in its interior, then $I$ contains infinitely 
many social boundary points. Then either $I$ contains three consecutive such
points, or contains two consecutive that form an equivalence class of size
two.
In both cases, at least one of them is white and at least one is black. We can
proceed similarly to find infinitely many points of both colors.
\end{proof}

\begin{defi}
A boundary point $p$ is called {em rich} if there is a translate of an
$S$-wedge $E_i$, such that $p$ is the only boundary point in it, but it
contains at least one interior point, see in Figure \ref{rich}.
\end{defi}

\begin{figure}[htpb]
\begin{center}
\includegraphics[height=30mm]{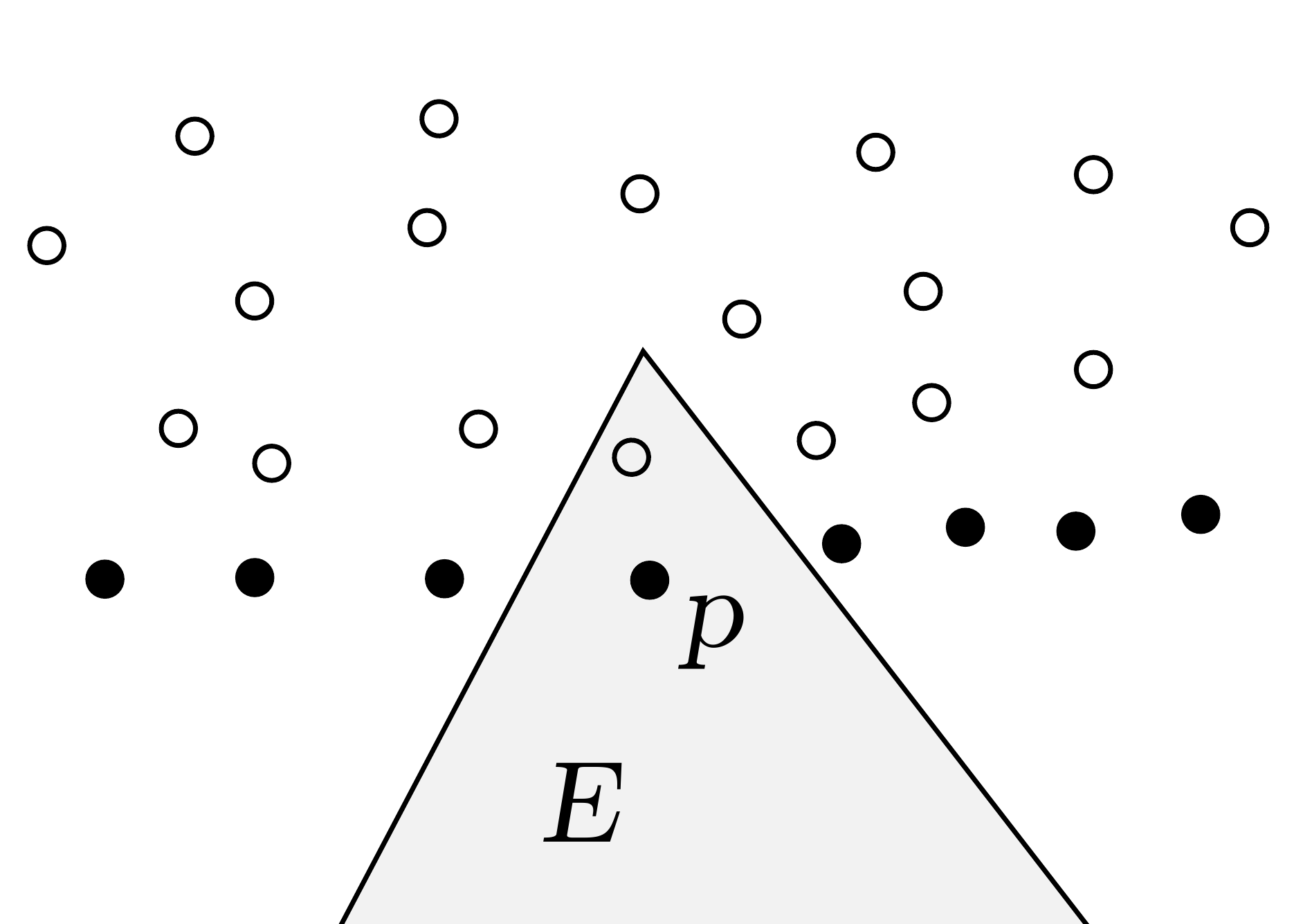}
\caption{The point $p$ is rich boundary point}
\label{rich}
\end{center}
\end{figure}
\subsection{Proof of Theorem \ref{dualfotetel}}

Now we are ready to prove Theorem \ref{dualfotetel}.
Suppose that $S$ is a closed, centrally symmetric convex polygon,
its vertices are $v_1, v_2, \ldots , v_{2n}$, ordered clockwise, the
$S$-wedges are $E_1$, $E_2$, $\ldots $, $E_{2n}$.
Let $S'$ be $S$ minus its boundary. 
Let ${\cal H}$ be a bounded set of points.

First we color the boundary, $\mathbb{B}$,
of ${\cal H}$, then we color the boundary 
 $\mathbb{B}'$ of the interior points, and finally we color the remaining points.
Very roughly speaking, the first level will be ``responsible'' for color blue
in wedges which contain many, but finitely many points, the next level is
responsible for color red, and coloring of the remaining points settles the
wedges with infinitely many points.

\medskip

{\sc Red-Blue-Coloring$(S, {\cal H})$}

\smallskip

\begin{itemize}

{\sl 

\item[]1. Let  $\mathbb{B}$ be the boundary of ${\cal H}$ with respect to
$S$. Color  $\mathbb{B}$, the {\em first level}, with 
procedure {\sc Black-White-Boundary-Coloring($S, {\cal H}$)}.
Then, let $p\in\mathbb{B}$ be 
\begin{itemize}
	\item blue, if rich or white,
	\item red otherwise.
\end{itemize}
Now let  ${\cal H}'={\cal H}\setminus \mathbb{B}$, the set of interior points.

\smallskip

2.  Let  $\mathbb{B}'$ be the boundary of ${\cal H}'$ with respect to
$S$. Color  all points of $\mathbb{B}'$, the {\em second level}, red.
%
Let  ${\cal H}''={\cal H}'\setminus \mathbb{B}'$, the set of interior
points of  ${\cal H}'$.

\smallskip

3.  Let  $\mathbb{B}''$ be the third level, boundary of ${\cal H}''$ with respect to
$S'$. (Watch out, $S'$ and not $S$!)
Color $\mathbb{B}''$  with procedure  {\sc Black-White-Boundary-Coloring($S',
  {\cal H}''$)}.
Then, let $p\in\mathbb{B}''$ be 
\begin{itemize}
	\item blue, if white,
	\item red if black.
\end{itemize}
Finally, let ${\cal H}'''={\cal H}''\setminus \mathbb{B}''$, the fourth level,
set of interior
points of  ${\cal H}''$, the set of still uncolored points.

\smallskip

4. Take a square that contain 
 ${\cal H}''$.
If it contains finitely many points  of ${\cal H}''$ (that is, ${\cal H}''$
has finitely many points),
color
them red and stop.
If it contains infinitely many points of ${\cal H}''$, 
then color one red and one blue. 
Divide the square into four smaller squares. In each of them, if there are
finitely many points of ${\cal H}''$, then color all uncolored points red, and
do not consider this square anymore.
If there are infinitely many  points in it, then color an uncolored point red
and another one blue. Then divide it into four little squares.
Repeat this infinitely many times. 
Finally color all points, which are still uncolored, red.
}
\end{itemize}

Now we prove that this coloring satisfies the conditions.
Suppose that $E_i(a)$ contains {\em finitely many} points of
${\cal H}$, but at least $9$. Then $E_i(a)$ contains at least one boundary
point of ${\cal H}$.

\smallskip

Case 1. $E_i(a)$ contains one point from the first level, that is,
$|E_i(a)\cap\mathbb{B}|=1$.
Then this point is rich, so it is blue, and 
 $E_i(a)$ contains at least 8 interior points.

\smallskip

Case 2. $|E_i(a)\cap\mathbb{B}|=2$. By Claim \ref{2interval}, $E_i(a)$
intersects  $\mathbb{B}$ in at most two intervals. If both contain one point,
then at least one of them is rich, so it is blue, and $E_i(a)$ contains at least 7 interior points.

\smallskip

Case 3. $3\le |E_i(a)\cap\mathbb{B}|\le 8$. Since $E_i(a)$
intersects  $\mathbb{B}$ in at most two intervals, it contains two consecutive 
boundary points, so one of them is blue, and $E_i(a)$ contains at least one interior point.

\smallskip

Case 4.  $|E_i(a)\cap\mathbb{B}|\le 9$. Then  $E_i(a)$ contains at
least 5 consecutive boundary points, say, $p_1$, $p_2$, $\ldots $, $p_{5}$.
At least three of them are blue. Suppose that all of them are blue. Procedure 
 {\sc Black-White-Boundary-Coloring($S, {\cal H}$)}
did not color three consecutive points white, therefore, at least one of
$p_2$, $p_3$ and $p_4$ got color blue, because it is rich. 
It is not hard to see that $E_i(a)$ contains
the
interior points corresponding to this rich boundary point. 

\smallskip

Summarizing, if  $E_i(a)$ contains at least $9$ but finitely many points,
then either it contains a point of both colors, or 
it contains a blue point on the boundary, and at least one interior point.
But in this case it contains a point of $\mathbb{B}'$, the  boundary of the
interior points, 
which is red, so we are done in the case when 
$E_i(a)$ contains {\em finitely many} but at least 9 points of
${\cal H}$.


\medskip

Now suppose that $E_i(a)$ contains {\em infinitely many} points of ${\cal H}$,
and suppose for contradiction that it does not contain a point of both colors.

\smallskip

Case 1. $E_i(a)$ contains infinitely many points from the boundary 
of ${\cal H}$. By Claim \ref{2interval}, $E_i(a)\cap\mathbb{B}$ consists of 
of at most two intervals, one of the intervals, say $I$, is infinite.
Procedure 
{\sc Black-White-Boundary-Coloring($S, {\cal H}$)} 
colors infinitely many points of $I$ to both colors.
It follows immediately, that there are infinitely many blue points in 
 $E_i(a)$.
Therefore, by our assumption, all points of $I$ got color blue. Then infinitely many of them
are rich. But then the infinitely many interior points that correspond to
these rich points, are also in 
 $E_i(a)$.

\smallskip

Case 2. $E_i(a)$ contains finitely many points from the boundary 
of ${\cal H}$, but at least one.
Then, just like in the finite case, it is not hard to see that
 $E_i(a)$ contains at least one blue point from the boundary, and infinitely
 many 
interior points.

\smallskip

Case 3.  $E_i(a)$ does not contain boundary points. 
Obviously, it contains infinitely many interior points, and by the definition it doesn't contain boundary points of  ${\cal H}'$.

\smallskip

So we can conclude that $E_i(a)$ contains infinitely many interior
points, and either it contains a blue boundary point, or no boundary points at
all.
Since 
we colored the boundary of the interior points, $\mathbb{B}'$ red,
we obtain that $E_i(a)$ contains a red point from its boundary, or no boundary
points at all, and infinitely many interior points of  ${\cal H}'$.

We assumed that  $E_i(a)$  does not contain a point of both colors,
therefore, either $E_i(a)\cap \mathbb{B}=\emptyset$ (and therefore  $E_i(a)\cap \mathbb{B}'=\emptyset$),
or  $E_i(a)\cap \mathbb{B}'=\emptyset$.
Assume the first, the argument in the second case is the same.

We know that $E_i(a)$ contains infinitely many points from
${\cal H}''$, the set of interior points of ${\cal H}'$.
We distinguish two cases.

\smallskip

Case 1. $E_i(a)$ contains infinitely many points from $\mathbb{B}''$, the
boundary of ${\cal H}''$ with respect to $S'$. The set 
 $E_i(a)\cap \mathbb{B}''$ is again the union of at most two intervals,
 therefore, one of the intervals contain infinitely many points, so by Claim
 \ref{mindketszin} it contains infinitely many points of both colors.

\smallskip

Case 2. $E_i(a)$ contains finitely many points from $\mathbb{B}''$.
Then it contains  infinitely many points from  the set ${\cal H}'''$, the
interior points of  ${\cal H}''$, with respect to $S'$.
We claim that in this case 
 $E_i(a)$ contains a point in its {\em interior}.
Suppose not. Then all points in  $E_i(a)\cap  {\cal H}''$ are on the boundary
of $E_i(a)$, so they all belong to $\mathbb{B}''$, a contradiction.
Therefore, there is a point $a_0\in{\cal H}$ in the interior of 
 $E_i(a)$.
Clearly, $E_i(a_0)$ is also in the interior of  $E_i(a)$. $E_i(a)\cap
\mathbb{B}=\emptyset$, hence 
$a_0$ is not a boundary point of ${\cal H}$, so 
there is a point $a_1$ in  $E_i(a_0)$. Since $a_1$ is not a boundary point
either, 
there is an $a_2$ in  $E_i(a_1)$. This way we get an infinite sequence
$a_0, a_1, \ldots $ of points in $E_i(a_0)$. With the exception of finitely
many, they belong to ${\cal H}'''$. They have an accumulation point 
$x$. The point 
$x\in E_i(a_0)$ since  $E_i(a_0)$ is closed, so $x$ is in the interior of 
$E_i(a)$.
Therefore, when we colored  ${\cal H}'''$, in step 4 of procedure 
{\sc Red-Blue-Coloring$(S, {\cal H})$}, once we arrived to a little square
which is in $E_i(a)$, 
contains $x$, 
and contains infinitely many points. So we colored one of the blue and one of
them red.
This concludes the proof of Theorem \ref{dualfotetel}.

\section{Infinite-fold coverings; Proof of Theorem 2.}


Just like in the proof of Theorem \ref{fotetel},
we can take the dual of the problem, and divide the plane into small 
squares. Therefore, it is enough to prove the following 
result.

\begin{thm}\label{vegtelen-dual}
Let $S$ be a closed, convex, centrally symmetric polygon, its vertices are
$v_1$, $v_2$, $\ldots ,$ $v_{2n}$, oriented clockwise. Then 
any bounded point set $\cal H$ can be colored with red and blue such that for
any translate of an $S$-wedge $E_i(p)$, if 
 $|E_i(p)\cap {\cal H}|=\infty $, then $E_i(p)\cap {\cal H}$
contains infinitely many red and infinitely many blue points. 
\end{thm}

Let $S'$ be $S$ minus its boundary and for $1\le i\le 2n$, let 
$E_i'$ be $E_i$ minus its boundary. That is, $E'_1$, $E'_2$, $\ldots$,
$E'_{2n}$
are the $S'$-wedges.

From now on, {\em boundary} of a point set is understood according to
$S'$, and not $S$. 

\begin{itemize}
\item Let $\mathbb{B} = \mathbb{B}^{(0)}$ the set of boundary points
          of $\cal H$. Its interior points 
${\cal H}^{(1)} = {\cal H} \setminus \mathbb{B}^{(0)}$.

\item Let $\mathbb{B}' = \mathbb{B}^{(1)}$ the set of boundary points
          of ${\cal H}^{(1)}$. Its interior points 
${\cal H}^{(2)} = {\cal H}^{(1)} \setminus \mathbb{B}^{(1)}$.
  
\item $\ldots $
  
\item Let $\mathbb{B}^{(n)}$ be the set of boundary points
          of ${\cal H}^{(n)}$. Its interior points  
${\cal H}^{(n+1)} = {\cal H}^{(n)} \setminus \mathbb{B}^{(n)}$.

\item $\ldots $

\end{itemize}

Moreover, let 
$\mathbb{B}^* = \bigcup_{n \in \mathbb{N}} \mathbb{B}^{(n)}$, 
and  ${\cal H}^* = {\cal H} \setminus \mathbb{B}^*$.

We call  $\mathbb{B}^{(n)}$ the $n$-th boundary level of ${\cal H}$.
Let $\mathbb{B}_i^{(n)}$ be the set of $E_i'$-boundary points 
of $\mathbb{B}^{(n)}$, and let 
$\mathbb{B}_i^*=\bigcup_{n \in \mathbb{N}} \mathbb{B}_i^{(n)}$.


Now we are ready to give the coloring algorithm.

\bigskip

{\sc Multiple-Red-Blue-Coloring($S, {\cal H}$)}


\begin{itemize}

\medskip


{\sl 


\item[] Step 1. We color a subset of  $\mathbb{B}^*$ so that we color at most one point
from each four consecutive levels.
For each $p \in \mathbb{B}^*$ let $h(p) = n$ if and only if $p \in
\mathbb{B}^{(n)}$. That is, each $p$ is on the $h(p)$-th level.
Take a square $Q_1$ which contains 
$\mathbb{B}^*$. Divide it into four little squares,
these are $Q_2$, $Q_3$, $Q_4$ and $Q_5$. Then divide $Q_2$ into four little
squares, these are $Q_6$, $Q_7$, $Q_8$, $Q_9$. 
Similarly, divide $Q_3$ to get  $Q_{10}, \ldots , Q_{13}$, and continue
similarly.
Eventually we divide each square in the list into four little squares,
and put them in the list.
This way we obtain an infinite list $Q_1$, $Q_2$, $\ldots $
of squares.
%
%

In Step 1.1, if $Q_1$ contains infinitely many points of $\mathbb{B}^*$,
then color one of them, $p_1$, {\bf red}. Otherwise, we stop.
In Step 1.2, if $Q_1$  contains infinitely many points of
 $\mathbb{B}^* \setminus \bigcup_{l< h(p_1) + 3 } \mathbb{B}^{(l)}$,
then color one of them, $p_2$, {\bf blue}. Otherwise, we stop.

In general, in Step $1.(2k-1)$, if  $Q_k$  contains infinitely many points of
the set 
$\mathbb{B}^* \setminus
\bigcup_{l<h(p_{2k-1}) +3} \mathbb{B}^{(l)}$, 
then color one of them, $p_{2k-1}$, {\bf red}. Otherwise, we stop.
Then, in Step $1.2k$, if  $Q_k$  contains infinitely many points of
the set 
$\mathbb{B}^* \setminus
\bigcup_{l<h(p_{2k-2}) +3} \mathbb{B}^{(l)}$,
then color one of them, $p_{2k}$, {\bf blue}. Otherwise, we stop.

After countably many steps, we are done with Step 1.

\medskip

In the following steps we color the uncolored points. 

\medskip

\item[] Step 2.  For each even $n$, color 
$\mathbb{B}^{(n)}$ with procedure 
{\sc Black-White-Boundary-Coloring($S', {\cal H}^{(n)}$)}.
Now a boundary point $p\in\mathbb{B}^{(n)}$ will be 
\begin{itemize}
	\item {\bf blue}, if it is rich of white,
	\item {\bf red} otherwise.
\end{itemize}

\medskip

\item[] Step 3.  For each odd $n$, color 
$\mathbb{B}^{(n)}$ with procedure 
{\sc Black-White-Boundary-Coloring($S', {\cal H}^{(n)}$)}.
Now a boundary point $p\in\mathbb{B}^{(n)}$ will be 
\begin{itemize}
	\item {\bf red}, if it is rich of white,
	\item {\bf blue} otherwise.
\end{itemize}

That is, we change the roles of the colors.

\medskip

\item[] Step 4. Take a square which contains 
${\cal H}^*$. If
it contains infinitely many points from 
${\cal H}^*$, (that is, ${\cal H}^*$ has infinitely many points)
then color one of them blue and one of them red. 
Divide the square into four little squares.
In each of them, which contains infinitely many points from 
${\cal H}^*$, color one of the uncolored points blue and one of them red, and
divide it into four smaller squares.
Continue recursively. 
Once we obtain a square which contains only finitely many 
 points from 
${\cal H}^*$, color all uncolored points red, and
do not divide it into smaller squares.

\medskip

} 
\end{itemize}

Suppose that ${\cal H}$ is colored by procedure 
{\sc Multiple-Red-Blue-Coloring($S, {\cal H}$)}.
We show that if a translate of an $S$-wedge contains infinitely many 
points of ${\cal H}$, then it contains infinitely many 
points of both colors.
First we show that a wedge contains an accumulation point in its interior,
then it contains infinitely many 
points of both colors.

\begin{lem}
Suppose that 
$E_i(a) \cap {\cal H}$ is infinite and this set has 
an accumulation point in the interior of $E_i(a)$.
Then $E_i(a)$  contains  infinitely many 
points of both colors.
\end{lem}

\begin{proof}
We have several cases according to the types of points the converge to $t$.

{\bf 1.} Point $t$ is the accumulation point of the interior points
(${\cal H}^*$).
In this case, in Step 4, we found infinitely many little squares that contain
infinitely many points of ${\cal H}^*$ but contained in  $E_i(a)$. Therefore,
$E_i(a)$ contains 
infinitely many 
points of both colors.

{\bf 2.} There are infinitely many boundary levels whose points converge to
$t$. In this case we can argue similarly as in the previous case.
In Step 1 of the procedure we produce a red and a blue sequence of points 
that converge to $t$. $E_i(a)$ contains infinitely many of both sequences.

{\bf 3.} Suppose now that there are only finitely many  boundary levels whose points converge to
$t$, and $t$ is not the  accumulation point of the interior points
(${\cal H}^*$).
Let $n$ be the largest number with the property that $t$ is an 
accumulation point of $\mathbb{B}^{(n)}$. 

Then it follows from Claim \ref{mindketszin} that 
$E_i(a)$ contains infinitely many
black and white points. If $E_i(a)$ does not contain infinitely many 
red and blue points, then there is  a sequence $p_1, p_2, \ldots$ 
of rich boundary points that
converge to $t$.
For each rich boundary point $p_j\in \mathbb{B}^{(n)}$, there is a point $p'_j\in 
{\cal H}^{(n+1)}$ which ``proves its richness'', that is, there is a translate
$E^{(j)}_i$ of $E_i$ which contains $p_j$ and no other boundary point of 
${\cal H}^{(n)}$, and also contains 
$p'_j$ of 
${\cal H}^{(n+1)}$. Since the sequence $p_1, p_2, \ldots$ converges to $t$,
the distance between $p_j$ and $p_{j+1}$ also goes to $0$ as $j$ goes to
infinity.
Therefore, the distance between $p_j$ and $p'_{j}$ also goes to $0$, so 
the sequence $p'_1, p'_2, \ldots$ converges to $t$ as well.
But this contradicts our assumptions. 
\end{proof}

So, we are done if the points in $E_i(a)$ have an accumulation point in the
interior of it. Suppose now that there is no such accumulation point.

I. Assume that 
$E_i(a) \cap \mathbb{B}_i^*$ is infinite.
Observe that if 
$E_i(a)\cap \mathbb{B}_i^{(n)} \neq \emptyset$, then, by the definition of the
boundary levels, for every  $k<n$,
$E_i(a)\cap \mathbb{B}_i^{(k)} \neq \emptyset$.
%

We distinguish two subcases.

{\bf (a)} Suppose that for every $n$, $E_i(a)\cap \mathbb{B}_i^{(n)}\neq
\emptyset$. Then, since we changed the roles of the colors for the even and
odd numbered levels, 
for $n$ even, the sets $E_i(a)\cap \mathbb{B}_i^{(n)}$
contain infinitely many blue points, for $n$ odd, they contain infinitely many
red points. 

{\bf (b)} Suppose now, that $E_i(a)\cap
\mathbb{B}_i^{(n)} \neq \emptyset$ holds only for finitely many levels.
Let $n$ be the largest number such that 
$E_i(a)\cap \mathbb{B}^{(n)}$ is infinite.
By Claim \ref{mindketszin},
procedure 
{\sc Black-White-Boundary-Coloring($S, {\cal H}^{(n)}$)} colors 
infinitely many points of $E_i(a)\cap \mathbb{B}^{(n)}$ black and white.
So, the only problem could be, that infinitely many black point of them are
rich. Now we can argue similarly as in part 3. Let 
$p_1, p_2, \ldots$ be a sequence of 
of rich boundary points in $\mathbb{B}^{(n)}$. 
For each $p_j\in \mathbb{B}^{(n)}$, there is a point $p'_j\in 
{\cal H}^{(n+1)}$, and a translate of $E^{(j)}_i$ of $E_i$  which ``prove its richness''.
But then $E_i(a)$ also contain the sequence
$p'_1, p'_2, \ldots$. Since $n$ is the largest number such that 
$E_i(a)\cap \mathbb{B}^{(n)}$ is infinite, only finitely many of 
$p'_1, p'_2, \ldots$ could belong to  $\mathbb{B}_i^*$.
On the other hand, if any $E^{(j)}_i$ contains infinitely many 
points of ${\cal H}^{(n+1)}$, then they have an accumulation point which is in
the interior of  $E_i(a)$, contradicting our assumption. 
Therefore, each $E^{(j)}_i$ contains only finitely many 
points of ${\cal H}^{(n+1)}$. But then they all belong to 
some boundary level, which is a contradiction again.
%

\medskip

II. Finally, suppose that $E_i(a) \cap \mathbb{B}_i^*$ is finite. We can
assume without loss of generality that it is empty.
We assumed that there is no accumulation point in the interior of 
 $E_i(a)$.
If  $E_i(a)$ does not contain {\em any} point in its interior, we have a
contradiction, since in this case all points in  $E_i(a)$ are $E_i$-boundary points.
If it contains a point $p_0$ in its interior, then, since it is not an
$E_i$-boundary point, $E_i(p_0)$ contains a point $p_1$. Since it is not an
$E_i$-boundary point either, $E_i(p_1)$ also contains a point $p_2$. 
We get an infinite sequence $p_1, p_2, \ldots$ in  $E_i(p_0)$, so they have an
accumulation point $t$ in the interior of $E_i(a)$. It is again a
contradiction. 

This concludes the proof of Theorem \ref{vegtelen-dual}, and therefore we also
proved Theorem \ref{vegtelen}.


\section{Remarks; Other versions of cover-decomposability}

1. The concept of cover-decomposability has many other versions, instead of the
plane, we can consider multiple coverings of an arbitrary set, 
we can assume that we have finitely many, countably many, or arbitrarily many 
translates in the covering. These versions are sometimes confused in the
literature, moreover, there are some incorrect statements because not the
correct version of  cover-decomposability is used.
See \cite{P10} and \cite{PPT14} for an overview.

Every covering in the sequel is a family of
translates of a planar set $S$.

\begin{defi}

{\rm (a)} A covering is {\em finite}, if it contains finitely many translates.

{\rm (b)} A covering is  {\em locally finite}, if any compact set intersects
only finitely many translates.

{\rm (c)} A covering is {\em countable}, if it contains countably many translates.
\end{defi}

Now we define eight different versions of cover-decomposability.

\begin{defi}
A planar set $S$ is

\noindent $\{$finite, locally finite, countable, or arbitrary$\}$ 

\noindent $\{$plane- or total-$\}$
 
\noindent cover-decomposable, if there is a constant $k$ such that any

\noindent $\{$finite, locally finite, countable, or arbitrary$\}$ 

\noindent $k$-fold covering of the 
 
\noindent $\{$the plane, or any planar set$\}$

\noindent can be decomposed into two coverings.

\end{defi}

Our Theorem \ref{fotetel} states that every 
centrally symmetric
closed convex polygon is plane-arbitrary-cover-decomposable.
It is not hard to see, that our proof works also for the other versions of
cover-decomposability.
It was known only for those versions which could be reduced to a finite
problem.
The next table summarizes the references for all positive results for 
centrally symmetric
closed convex polygons. [KT] refers to the present note.

\bigskip
{\footnotesize
\begin{tabular}{c|c|c|c|c|}
 & finite
 & locally finite
 & countable
 & arbitrarily many \\
\hline
 a plane
 & $-$ & \cite{P86} & [KT] & [KT] \\
\hline
 any planar set
 & \cite{P86} & \cite{P86} & [KT] & [KT] 
\end{tabular}}

\bigskip

Our proof, with hardly any modification, implies the same results for {\em open}
centrally symmetric
convex polygons. 
In this case it is easier to reduce the problem to the finite case, therefore, 
cover-decomposability was proved for more versions.
The next table summarizes the situation for 
centrally symmetric
open convex polygons. 

\bigskip
{\footnotesize
\begin{tabular}{c|c|c|c|c|}
 & finite
 & locally finite
 & countable
 & arbitrarily many \\
\hline
 a plane
 & $-$ & \cite{P86} & \cite{P86} & \cite{P86} \\
\hline
 any planar set
 & \cite{P86} & \cite{P86} & [KT] & [KT] 
\end{tabular}}

\bigskip

2. In the proof of Theorems \ref{fotetel} and \ref{vegtelen} 
we used different colorings. In fact, there is a single 
coloring algorithm which
can be used in both proofs, but we found it too technical to present it.

\end{document}